\documentclass[12pt, a4paper]{article}
\usepackage{graphicx}
\usepackage{amssymb}
\usepackage{stmaryrd}
\usepackage{enumitem}

\usepackage{geometry}
\usepackage{tikz}
\usepackage{float}
\usepackage{url}
\usepackage{mathtools}
\usepackage{amsmath,amsfonts,amsthm}
\usepackage{mathrsfs}

\newtheorem{theorem}{Theorem}[section]
\newtheorem{lemma}[theorem]{Lemma}
\newtheorem{proposition}{Proposition}[section]
\newtheorem{corollary}[theorem]{Corollary}
\newtheorem{definition}{Definition}

\newtheorem{remark}{Remark}

\usepackage{mma}
\usepackage{ytableau}
\usepackage[colorlinks,
linkcolor=blue,
anchorcolor=blue,
citecolor=blue
]{hyperref}
\begin{document}
\begin{center}
{\large \bf Irregularity and Topological Indices in Fibonacci Word Trees and Modified Fibonacci Word Index}
\end{center}
\begin{center}
 Jasem Hamoud$^1$ and Duaa Abdullah$^{2}$\\[6pt]
 $^{1,2}$ Physics and Technology School of Applied Mathematics and Informatics \\
Moscow Institute of Physics and Technology, 141701, Moscow region, Russia\\[6pt]

Email: 	 $^{1}${\tt khamud@phystech.edu},
 $^{2}${\tt abdulla.d@phystech.edu}
\end{center}
\noindent
\textbf{Abstract.}
This paper introduces the concept of the Fibonacci Word Index $\operatorname{FWI}$, a novel topological index derived from the Albertson index, applied to trees constructed from Fibonacci words. Building upon the classical Fibonacci sequence and its generalizations, we explore the structural properties of Fibonacci word trees and their degree-based irregularity measures. We define the $\operatorname{FWI}$ and its variants, including the total irregularity and modified Fibonacci Word Index where it defined as
\[
\operatorname{FWI}^*(\mathscr{T})=\sum_{n,m\in E(\mathscr{T})}[\deg F_n^2-\deg F_m^2],
\]
and establish foundational inequalities relating these indices to the maximum degree of the underlying trees. Our results extend known graph invariants to the combinatorial setting of Fibonacci words, providing new insights into their algebraic and topological characteristics. Additionally, we present analytical expressions involving Fibonacci numbers and their generating functions, supported by Binet’s formula, to facilitate computation of these indices. The theoretical developments are illustrated with examples, including detailed constructions of Fibonacci word trees and their degree distributions. This work opens avenues for further investigation of word-based graph invariants and their applications in combinatorics and theoretical computer science.

\noindent\textbf{AMS Classification 2010:} \textbf{Primary}: 05C05, 05C12, 05A15, 68R15. \textbf{Secondary}: 11B39, 05C90, 68Q70.


\section{Introduction}\label{sec1}
Throughout this paper, let $\mathscr{T}$ be a tree of Fibonacci word $F$, the Fibonacci sequence, recursively defined by $f_{1}=f_{2}=1$ and $f_{n}=$ $f_{n-1}+f_{n-2}$ for all $n \geq 3$, has been generalised in several ways. In 2000, Divakar Viswanath in~\cite{Viswanath} studied random Fibonacci sequences given by $t_{1}=t_{2}=1$ and $t_{n}= \pm t_{n-1} \pm t_{n-2}$ for all $n \geq 3$. Here each $\pm$ is chosen to be + or with probability $1 / 2$, and are chosen independently. Viswanath proved that
\[
\lim _{n \rightarrow \infty} \sqrt[n]{\left|t_{n}\right|}=1.13198824 \ldots
\]
The Fibonacci word is a sequence of binary strings over the alphabet $\{0, 1\}$, constructed recursively in a manner analogous to the Fibonacci sequence but using string concatenation instead of numerical addition. It is defined as $S_0 = 0$, $S_1 = 01$, and $S_n = S_{n-1} S_{n-2}$ for $n \geq 2$, where concatenation combines the strings, furthermore see~\cite{BerstelQWE,lenr,MonnerotDumaine}. In 1997,  M. O. Albertson in~\cite{ALBERTSON} mention to the imbalance of an edge $uv$ by $imb(uv)=\lvert d_u-d_v\rvert$, we considered graph $G$ regular if all of its vertices have the exact same degree, then irregularity measure define in~\cite{ALBERTSON, GUTMAN2, Brandt} as: 
\[
\operatorname{irr}(G)=\sum_{uv\in E(G)}\lvert d_u(G)-d_v(G) \rvert.
\]
Let $M_1(\mathscr{T})$ and $M_2(\mathscr{T})$  be the first and the second Zagreb index of Fibonacci word are defined (furthermore see~\cite{TRINAJSTIC, WILCOX,ALBERTSON}) as: 
\[
M_1(\mathscr{T})=\sum_{i=1}^{n}\deg F_i^2, \quad \text{and} \quad M_2(G)=\sum_{n,m\in \mathbb{N}} \deg_{F_n}(\mathscr{T})\deg_{F_m}(\mathscr{T}).
\]
\begin{definition}
The length of the word $|\mathbf{w}|$ is the number of letters contained in it. The empty word is denoted by $\varepsilon$.
\end{definition}
\begin{definition}
Infinite words as functions $\mathbf{w}: \mathbb{N} \rightarrow \Sigma$. The set of all finite words over $\Sigma$ is denoted by $\Sigma^{*}$, and $\Sigma^{+}=\Sigma^{*} \backslash\{\varepsilon\}$; the set of all infinite words is denoted by $\Sigma^{\mathbb{N}}$.
\end{definition}
In 2013, Ramírez, J.L, et al. In~\cite{Ramírez} mention to $k$-Fibonacci Words  are an extension of the Fibonacci word notion that generalises Fibonacci word features to higher dimensions. These words were investigated for their distinct curves and patterns. Recently, in 2023 Rigo, M., Stipulanti, M., \& Whiteland, M.A. in~\cite{R3} mentioned the Thue-Morse sequence, which is the fixed point of the substitution $0\rightarrow 01, 1\rightarrow 10$, has unbounded 1-gap $k$-binomial complexity for $k \geq 2$. Also,  we want to mention for a Sturmian sequence and $g\geq 1$, all of its long enough factors are always pairwise $g-gap k$-binomially inequivalent for any $k \geq 2$. 
Furthermore, for Fibonacci sequence with trees see in~\cite{ref02,R3,path7}. \par
This paper is organized as follows. In Section~\ref{sec1}, we observe the important concepts to our work including literature view of most related papers, in Section~\ref{sec2} We have provided a preface through some of the important theories and properties we have utilised in understanding the work, in Section~\ref{sec3} we introduced the main result of our paper and the goal of this paper about Fibonacci word, in Section~\ref{sec4} we computing the advanced result in density of Fibonacci word.

\section{Preliminaries}~\label{sec2}
The infinite Fibonacci word, obtained as the limit $n \to \infty$, is a Sturmian word with a slope related to the golden ratio $\phi = \frac{1 + \sqrt{5}}{2}$. This construction makes the Fibonacci word a significant object in combinatorics, with applications in coding theory and the study of non-repetitive sequences. The inclusion of negative indices is meant to simplify the computation of $F_{p,n}$ for $n=2,3,\dots,p$.
\begin{proposition}~\cite{PulidoRamírez}
For $1<n\leq p$, the term $F_{p,n}$ is given by $2^{n-2}$. 
\end{proposition}
\begin{definition}[Standard Fibonacci words~\cite{Kulkarni2021Mahalingam}]~\label{fibword01}
Let $\Sigma$ be an alphabet with $|\Sigma| \geq 2$ and let $u, v \in \Sigma^{+}$. The $n^{\text {th }}$ standard Fibonacci words are defined recursively as:
$$
\begin{gathered}
f_1(u, v)=u, f_2(u, v)=v, \\
f_n(u, v)=f_{n-1}(u, v) \cdot f_{n-2}(u, v), n \geq 3 .
\end{gathered}
$$
\end{definition}
A natural way to generalize this sequence is to consider a recursion where each term is the sum of the previous $p$ terms.   
For a positive integer $p$, the \emph{$p$-generalized Fibonacci numbers}, denoted by $F_{p,n}$, are  defined by the recursion:
    \[ F_{p,n}=    
        \sum_{i=1}^{p}F_{p,n-i}, \quad  n>1,
     \]
with the initial values $F_{p,n}=0$ for  $n=-p+2,-p+1,\dots, 0$, and $F_{p,n}=1$ for $n=1$. 
\begin{theorem}[Binet's Formula~\cite{knuth1997}]
The $n$-th Fibonacci number $F_n$ is given by:
\[
F_n = \frac{\phi^n - \psi^n}{\sqrt{5}},
\]
where $\phi = \frac{1 + \sqrt{5}}{2}$ (the golden ratio), $\psi = \frac{1 - \sqrt{5}}{2} = -\phi^{-1}$, and $n$ is a non-negative integer.
\end{theorem}
\begin{definition}[Correlation~\cite{Rampersad2023Wiebe}]
    For every pair of words $(u,v)\in A^n\times A^m$, the correlation of $u$ over $v$ is the word $C_{u,v}\in A^n$ such that for all $k\in [n]$, 
    \begin{align*} 
    C_{u,v}[k] =
    \begin{cases} 
        1 & \text{if }\; \forall i\in[n],\; j\in[m],\; \text{with } i = j + k, \text{ when } u[i] = v[j], \\
        0 & \text{otherwise.}
    \end{cases}
\end{align*}
\end{definition}
S. A. Hosseini et al. in~\cite{HosseiniGutmanAhmadi} was introduced the formal definition of a Kragujevac tree. Since all the researches~\cite{Furtula22Gutman,Furtula24Gutman,Furtula26Gutman} were done at the University of Kragujevac. Let $B_1,\dots,B_k$ be a branches of tree given in~\cite{HosseiniGutmanAhmadi}.
\begin{definition}
A proper Kragujevac tree is a tree possessing a central vertex of degree at least 3, to which branches of the form $B_1$ and/or $B_2$ and/or $B_3$ and/or ...are attached. 
\end{definition}
\begin{proposition}~\label{proex1}
Let $F_n$ be a Fibonacci word, and let $k>0$ an integer number, then
\[
\sum_{k=1}\frac{F_{2n+k}F_{2n-k}}{k!.2^n}=\dfrac{L_{4n} (e - 1) - (e^{-\phi^2} - 1) - (e^{-\phi^{-2}} - 1)}{5 \cdot 2^n}
\]
where $L_{4n} = F_{4n-2} + F_{4n+2},\phi = \frac{1 + \sqrt{5}}{2}$.
\end{proposition}
\begin{proof}
Let $F_n$ be the $n$-th Fibonacci word. Typically refers to the $n$-th Fibonacci number, defined by $F_0 = 0$, $F_1 = 1$ and $F_n = F_{n-1} + F_{n-2}$ for $n \geq 2$. Fibonacci words are sequences over an alphabet (e.g., {0,1}) generated by a substitution rule, but they are less commonly used in such analytical sums. Given the structure of the sum, involving factorials and numerical operations, it’s reasonable to interpret $F_n$ as the Fibonacci number sequence. Thus, we proceed assuming $F_n$ is the Fibonacci number. The Fibonacci numbers by using Binet’s formula: $F_m = \frac{\phi^m - \psi^m}{\sqrt{5}}$, where $\phi = \frac{1 + \sqrt{5}}{2}$ and $\psi = \frac{1 - \sqrt{5}}{2} = -\frac{1}{\phi}$, with $\phi + \psi = 1$, $\phi - \psi = \sqrt{5}$, and $\phi \psi = -1$. \par 
\noindent
Using Binet’s formula, we write: 
\[
F_{2n+k} = \frac{\phi^{2n+k} - \psi^{2n+k}}{\sqrt{5}}, \quad F_{2n-k} = \frac{\phi^{2n-k} - \psi^{2n-k}}{\sqrt{5}}.
\]
Thus, 
\[
F_{2n+k} F_{2n-k} = \frac{(\phi^{2n+k} - \psi^{2n+k})(\phi^{2n-k} - \psi^{2n-k})}{5}.
\] 
Expanding the numerator as $(\phi^{2n+k} - \psi^{2n+k})(\phi^{2n-k} - \psi^{2n-k}) = \phi^{4n} - \phi^{2n+k} \psi^{2n-k} - \phi^{2n-k} \psi^{2n+k} + \psi^{4n}$. Since $\phi^{2n+k} \psi^{2n-k} = (\phi \psi)^{2n-k} \phi^{2k} = (-1)^{2n-k} \phi^{2k} = (-1)^{k+1} \phi^{2k}$ (because $2n-k$ is odd when $k$ is odd and even when $k$ is even), and similarly for the other term, we need to handle the sum carefully. Instead, let’s compute the term: \[\frac{F_{2n+k} F_{2n-k}}{k! \cdot 2^n} = \frac{1}{5 \cdot k! \cdot 2^n} (\phi^{2n+k} - \psi^{2n+k})(\phi^{2n-k} - \psi^{2n-k}).\]
The sum becomes: \[\sum_{k=1}^\infty \frac{F_{2n+k} F_{2n-k}}{k! \cdot 2^n} = \frac{1}{5 \cdot 2^n} \sum_{k=1}^\infty \frac{(\phi^{2n+k} - \psi^{2n+k})(\phi^{2n-k} - \psi^{2n-k})}{k!}.\]
Expand the product inside the sum: $(\phi^{2n+k} - \psi^{2n+k})(\phi^{2n-k} - \psi^{2n-k}) = \phi^{4n} - \phi^{2n+k} \psi^{2n-k} - \phi^{2n-k} \psi^{2n+k} + \psi^{4n}$. So the summand is: $\frac{1}{k!} \left( \phi^{4n} - \phi^{2n+k} \psi^{2n-k} - \phi^{2n-k} \psi^{2n+k} + \psi^{4n} \right)$. We split the sum: $\sum_{k=1}^\infty \frac{\phi^{4n} - \phi^{2n+k} \psi^{2n-k} - \phi^{2n-k} \psi^{2n+k} + \psi^{4n}}{k!}$. Then
\[
\sum_{k=1}^\infty \frac{\phi^{4n} - (\phi / \psi)^k - (\psi / \phi)^k + \psi^{4n}}{k!} = \phi^{4n} (e - 1) - (e^{-\phi^2} - 1) - (e^{-\phi^{-2}} - 1) + \psi^{4n} (e - 1).
\]
Asb desire.
\end{proof}

\section{Main Result}~\label{sec3}
In this section, Fibonacci Word Index $\operatorname{FWI}$ is defined in Definition~\ref{defn1}, a novel topological index derived from the Albertson index, applied to trees constructed from Fibonacci words. Building upon the classical Fibonacci sequence and its generalizations. The maximum degree $\Delta$ of $\operatorname{FWI}$ in Proposition~\ref{pron1} defined as: $\Delta(\mathscr{T})=\{\max(\deg V(\mathscr{T}) \mid V\in \mathscr{T}\}$. 
\subsection{Fibonacci Word Index}~\label{subsec3}
In this subsection, we will introduce a ``Fibonacci Word Index'' according to the important topological index, it known a ``Albertson Index'' and its variants, including the total irregularity and modified Fibonacci Word Index, and establish foundational inequalities relating these indices to the maximum degree of the underlying trees. 
\begin{definition}~\label{defn1}
Let $F_n$, $F_m$ be a two Fibonacci word, let $\mathscr{T}$ be a tree of Fibonacci word, then Albertson index of Fibonacci word given us a Fibonacci Word Index, define by: 
\[
\operatorname{FWI}(\mathscr{T})=\sum_{n,m\in E(\mathscr{T})}\lvert \deg F_n-\deg F_m\rvert.
\]
\end{definition}
\noindent
The total irregularity of Fibonacci Word Index $\operatorname{FWI}_t(\mathscr{T})$ is given by: 
\[
\operatorname{FWI}_t(\mathscr{T})=\sum_{\{n,m\}\subseteq E(\mathscr{T})}\lvert \deg F_n-\deg F_m\rvert.
\]
The \textit{modified Fibonacci word index} $\operatorname{FWI}^*(\mathscr{T})$ is define as:
\[
\operatorname{FWI}^*(\mathscr{T})=\sum_{n,m\in E(\mathscr{T})}[\deg F_n^2-\deg F_m^2].
\]
Furthermore, see~\cite{YousafBhattiAli} for understanding clearly both of Proposition~\ref{pron1},\ref{pron2} and Lemma~\ref{lemn1}. The star $S_n$ graph of Fibonacci word defined by: 
\[
\operatorname{FWI}^*(S_n) = \sum_{u,v \in E(S_n)} \left( \deg(F_u)^2 - \deg(F_v)^2 \right) =\sum_{(u,v) \in E(S_n)} \left( \deg(u)^2 - \deg(v)^2 \right).
\]
\begin{proposition}~\label{pron1}
Let $\mathscr{T}$ be a tree with maximum degree $\Delta$ then $\operatorname{FWI}^*(\mathscr{T})\ge \Delta(\Delta^2 - 1)$ with equality if and only if $\mathscr{T}$ is isomorphic to either a path or a tree containing only one vertex of $\deg (V(\mathscr{T}))>2$.
\end{proposition}
\begin{proof}
Assume $\Delta\le2$, in this case, the result is clear and hence we assume that $\Delta\ge3$. If $v\in V(\mathscr{T})$ has maximum degree then there are $P$ pendant vertices namely $w_1,w_2,\cdots,w_{d_v}\,$ in $\mathscr{T}$ such that the paths $v-w_1,v-w_2,\cdots,v-w_{P}\,$ are pairwise internally disjoint.  Then,
\begin{align*}
&\big|(P_v)^2-(P_{w_{1,1}})^2\big| + \big|(P_{w_{1,1}})^2-(P_{w_{1,2}})^2\big| +\cdots + \big|(P_{w_{1,r}})^2 - (P_{w_{1}})^2\big|\\
&\ge \big[(P_v)^2-(P_{w_{1,1}})^2\big] + \big[(P_{w_{1,1}})^2-(P_{w_{1,2}})^2\big] +\cdots + \big[(P_{w_{1,r}})^2 - (P_{w_{1}})^2\big]=\Delta^2 -1\,.
\end{align*}
We note that the equality
\[
\big|(P_v)^2-(P_{w_{1,1}})^2\big| + \big|(P_{w_{1,1}})^2-(P_{w_{1,2}})^2\big| +\cdots + \big|(P_{w_{1,r}})^2 - (P_{w_{1}})^2\big| = \Delta^2 -1
\]
holds if and only if the degrees of successive vertices along the path from $v$ to $w_1$ decrease monotonously.
\end{proof}
Actually, Proposition~\ref{pron1} provided us a new results shown in Proposition~\ref{pron2} for the path $P_n$ and the star $S_n$ of tree $\mathscr{T}$. This results will be employee in Lemma~\ref{lemn1}. In fact, for extended more we show in Figure~\ref{fign1} Fibonacci word $F_{10}$ as an example of tree of Fibonacci word.

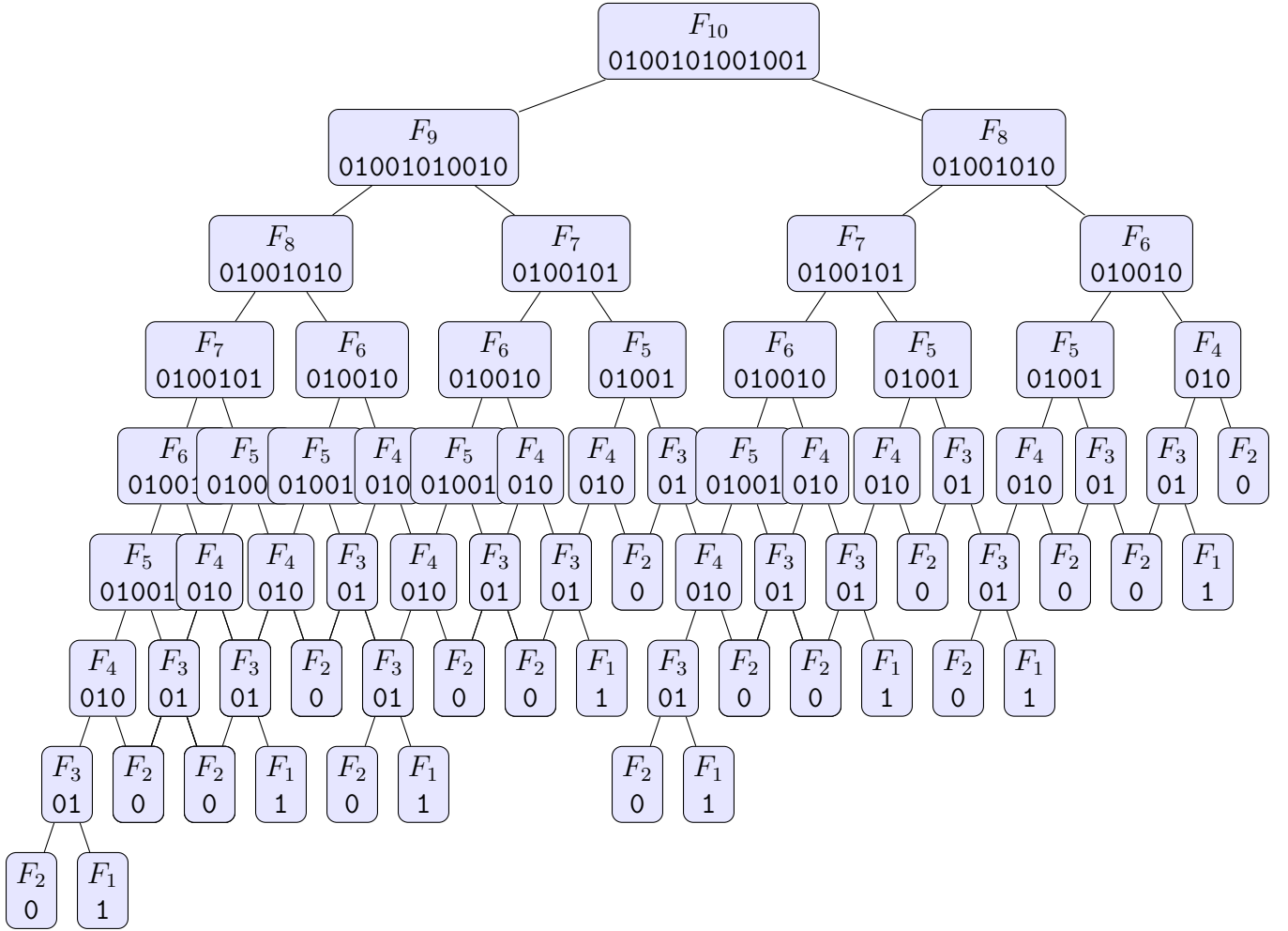
\begin{figure}[H]
    \centering
    \begin{tikzpicture}[
  level distance=1.5cm,
  level 1/.style={sibling distance=8cm},
  level 2/.style={sibling distance=4cm},
  level 3/.style={sibling distance=2cm},
  level 4/.style={sibling distance=1cm},
  every node/.style={rectangle, draw, rounded corners, align=center, fill=blue!10}
  ]

\node {\(F_{10}\)\\ \texttt{0100101001001}}
  child { node {\(F_9\)\\ \texttt{01001010010}}
    child { node {\(F_8\)\\ \texttt{01001010}}
      child { node {\(F_7\)\\ \texttt{0100101}}
        child { node {\(F_6\)\\ \texttt{010010}}
          child { node {\(F_5\)\\ \texttt{01001}}
            child { node {\(F_4\)\\ \texttt{010}}
              child { node {\(F_3\)\\ \texttt{01}}
                child { node {\(F_2\)\\ \texttt{0}} }
                child { node {\(F_1\)\\ \texttt{1}} }
              }
              child { node {\(F_2\)\\ \texttt{0}} }
            }
            child { node {\(F_3\)\\ \texttt{01}}
              child { node {\(F_2\)\\ \texttt{0}} }
              child { node {\(F_1\)\\ \texttt{1}} }
            }
          }
          child { node {\(F_4\)\\ \texttt{010}}
            child { node {\(F_3\)\\ \texttt{01}}
              child { node {\(F_2\)\\ \texttt{0}} }
              child { node {\(F_1\)\\ \texttt{1}} }
            }
            child { node {\(F_2\)\\ \texttt{0}} }
          }
        }
        child { node {\(F_5\)\\ \texttt{01001}}
          child { node {\(F_4\)\\ \texttt{010}}
            child { node {\(F_3\)\\ \texttt{01}}
              child { node {\(F_2\)\\ \texttt{0}} }
              child { node {\(F_1\)\\ \texttt{1}} }
            }
            child { node {\(F_2\)\\ \texttt{0}} }
          }
          child { node {\(F_3\)\\ \texttt{01}}
            child { node {\(F_2\)\\ \texttt{0}} }
            child { node {\(F_1\)\\ \texttt{1}} }
          }
        }
      }
      child { node {\(F_6\)\\ \texttt{010010}}
        child { node {\(F_5\)\\ \texttt{01001}}
          child { node {\(F_4\)\\ \texttt{010}}
            child { node {\(F_3\)\\ \texttt{01}}
              child { node {\(F_2\)\\ \texttt{0}} }
              child { node {\(F_1\)\\ \texttt{1}} }
            }
            child { node {\(F_2\)\\ \texttt{0}} }
          }
          child { node {\(F_3\)\\ \texttt{01}}
            child { node {\(F_2\)\\ \texttt{0}} }
            child { node {\(F_1\)\\ \texttt{1}} }
          }
        }
        child { node {\(F_4\)\\ \texttt{010}}
          child { node {\(F_3\)\\ \texttt{01}}
            child { node {\(F_2\)\\ \texttt{0}} }
            child { node {\(F_1\)\\ \texttt{1}} }
          }
          child { node {\(F_2\)\\ \texttt{0}} }
        }
      }
    }
    child { node {\(F_7\)\\ \texttt{0100101}}
      child { node {\(F_6\)\\ \texttt{010010}}
        child { node {\(F_5\)\\ \texttt{01001}}
          child { node {\(F_4\)\\ \texttt{010}}
            child { node {\(F_3\)\\ \texttt{01}}
              child { node {\(F_2\)\\ \texttt{0}} }
              child { node {\(F_1\)\\ \texttt{1}} }
            }
            child { node {\(F_2\)\\ \texttt{0}} }
          }
          child { node {\(F_3\)\\ \texttt{01}}
            child { node {\(F_2\)\\ \texttt{0}} }
            child { node {\(F_1\)\\ \texttt{1}} }
          }
        }
        child { node {\(F_4\)\\ \texttt{010}}
          child { node {\(F_3\)\\ \texttt{01}}
            child { node {\(F_2\)\\ \texttt{0}} }
            child { node {\(F_1\)\\ \texttt{1}} }
          }
          child { node {\(F_2\)\\ \texttt{0}} }
        }
      }
      child { node {\(F_5\)\\ \texttt{01001}}
        child { node {\(F_4\)\\ \texttt{010}}
          child { node {\(F_3\)\\ \texttt{01}}
            child { node {\(F_2\)\\ \texttt{0}} }
            child { node {\(F_1\)\\ \texttt{1}} }
          }
          child { node {\(F_2\)\\ \texttt{0}} }
        }
        child { node {\(F_3\)\\ \texttt{01}}
          child { node {\(F_2\)\\ \texttt{0}} }
          child { node {\(F_1\)\\ \texttt{1}} }
        }
      }
    }
  }
  child { node {\(F_8\)\\ \texttt{01001010}}
    child { node {\(F_7\)\\ \texttt{0100101}}
      child { node {\(F_6\)\\ \texttt{010010}}
        child { node {\(F_5\)\\ \texttt{01001}}
          child { node {\(F_4\)\\ \texttt{010}}
            child { node {\(F_3\)\\ \texttt{01}}
              child { node {\(F_2\)\\ \texttt{0}} }
              child { node {\(F_1\)\\ \texttt{1}} }
            }
            child { node {\(F_2\)\\ \texttt{0}} }
          }
          child { node {\(F_3\)\\ \texttt{01}}
            child { node {\(F_2\)\\ \texttt{0}} }
            child { node {\(F_1\)\\ \texttt{1}} }
          }
        }
        child { node {\(F_4\)\\ \texttt{010}}
          child { node {\(F_3\)\\ \texttt{01}}
            child { node {\(F_2\)\\ \texttt{0}} }
            child { node {\(F_1\)\\ \texttt{1}} }
          }
          child { node {\(F_2\)\\ \texttt{0}} }
        }
      }
      child { node {\(F_5\)\\ \texttt{01001}}
        child { node {\(F_4\)\\ \texttt{010}}
          child { node {\(F_3\)\\ \texttt{01}}
            child { node {\(F_2\)\\ \texttt{0}} }
            child { node {\(F_1\)\\ \texttt{1}} }
          }
          child { node {\(F_2\)\\ \texttt{0}} }
        }
        child { node {\(F_3\)\\ \texttt{01}}
          child { node {\(F_2\)\\ \texttt{0}} }
          child { node {\(F_1\)\\ \texttt{1}} }
        }
      }
    }
    child { node {\(F_6\)\\ \texttt{010010}}
      child { node {\(F_5\)\\ \texttt{01001}}
        child { node {\(F_4\)\\ \texttt{010}}
          child { node {\(F_3\)\\ \texttt{01}}
            child { node {\(F_2\)\\ \texttt{0}} }
            child { node {\(F_1\)\\ \texttt{1}} }
          }
          child { node {\(F_2\)\\ \texttt{0}} }
        }
        child { node {\(F_3\)\\ \texttt{01}}
          child { node {\(F_2\)\\ \texttt{0}} }
          child { node {\(F_1\)\\ \texttt{1}} }
        }
      }
      child { node {\(F_4\)\\ \texttt{010}}
        child { node {\(F_3\)\\ \texttt{01}}
          child { node {\(F_2\)\\ \texttt{0}} }
          child { node {\(F_1\)\\ \texttt{1}} }
        }
        child { node {\(F_2\)\\ \texttt{0}} }
      }
    }
  };

\end{tikzpicture}
    \caption{An example of Tree $\mathscr{T}$ of Fibonacci word for base $F_{10}$.}
    \label{fign1}
\end{figure}

\begin{proposition}~\label{pron2}
Let $P_n(\mathscr{T})$ be the path of a tree $\mathscr{T}$ of $n$-vertex where $n\geq 5$, let $S_n(\mathscr{T})$ be the star of $\mathscr{T}$, then $\operatorname{FWI}^* (P_n) < \operatorname{FWI}^*(\mathscr{T}) <\operatorname{FWI}^*(S_n)$.
\end{proposition}
Let $\alpha$ be a fixed vertex of $\mathscr{T}$. Let $N(\alpha)$, be a set of fixed number let $L(\alpha)$, $E(\alpha)$, $R(\alpha)$ be a sets of fixed vertices in $\mathscr{T}$ according to base $N(\alpha)$, where
\[
\begin{cases}
  L(\alpha)=\{\beta\in N(\alpha):~\deg(F_\beta)<\deg(F_\alpha)\}, \\
  E(\alpha)=\{\beta\in N(\alpha):~\deg(F_\beta)=\deg(F_\alpha)\}, \\
  R(\alpha)=\{\beta\in N(\alpha):~\deg(F_\beta)>\deg(F_\alpha)\}.
\end{cases}
\]
\begin{proposition}
 The number of elements in $L(\alpha)$, $E(\alpha)$ and $R(\alpha)$ are denoted by $l_\alpha$, $e_\alpha$ and $r_\alpha\,$, respectively. Clearly, $\deg(F_\alpha)=l_\alpha+e_\alpha+r_\alpha$.
\end{proposition}
\begin{lemma}~\label{lemn1}
Let $\mathscr{T}$ be a tree of Fibonacci word, let $\alpha$ and $\beta$ are non-adjacent vertices in $\mathscr{T}$ where $\deg(F_\alpha\ge \deg (F_\beta)$ then
\[
\operatorname{FWI}^*(\mathscr{T})=\operatorname{FWI}^*(\mathscr{T}-\alpha\beta)+2[(2\deg(F_\alpha)+1)r_\alpha+(2\deg(F_\beta)+1)r_\beta]-\Gamma(\mathscr{T}),
\]
where $\Gamma(\mathscr{T})=3\deg(F_\alpha)(\deg (F_\alpha)+1)-  \deg(F_\beta)(\deg (F_\beta)+1)$.
\end{lemma}
\begin{proof}
Firstly, we provided a motivation to stabilise the Fibonacci word by Proposition~\ref{pron1},\ref{pron2} as $\operatorname{FWI}^*(\mathscr{T})\ge \Delta(\Delta^2 - 1)$, then we have:  

\begin{align*}
\operatorname{FWI}^*(\mathscr{T}-\alpha\beta) - \operatorname{FWI}^*(\mathscr{T})
&=(\deg(F_\alpha)+1)^2 - (\deg(F_\beta)+1)^2 \\
& + \sum_{x\in N(\alpha)} \bigg( \big|(\deg(F_\alpha)+1)^2 - (\deg(F_x))^2\big| - \big|(\deg(F_\alpha))^2 - (\deg(F_x))^2\big|\bigg)\\
& + \sum_{y\in N(\beta)} \bigg( \big|(\deg(F_\beta)+1)^2 - (\deg(F_y))^2\big| - \big|(\deg(F_\beta))^2 - (\deg(F_y))^2\big|\bigg)\,.
\end{align*}
Now, consider $N(\alpha)=L(\alpha)\cup E(\alpha)\cup R(\alpha)$. We aim to determine the union of the sets $ L(\alpha) $, $ E(\alpha) $, and $ R(\alpha) $, defined as follows:
\[
\begin{cases}
L(\alpha) = \{\beta \in N(\alpha) : \deg(F_\beta) < \deg(F_\alpha)\}, \\
E(\alpha) = \{\beta \in N(\alpha) : \deg(F_\beta) = \deg(F_\alpha)\}, \\
R(\alpha) = \{\beta \in N(\alpha) : \deg(F_\beta) > \deg(F_\alpha)\},
\end{cases}
\]
where $ N(\alpha) $ is a set containing elements $ \beta $, and $ F_\beta $ and $ F_\alpha $ are polynomials associated with elements $ \beta $ and $ \alpha $, respectively, with $ \deg(\cdot) $ denoting the degree of a polynomial. The objective is to compute $ L(\alpha) \cup E(\alpha) \cup R(\alpha) $. The sets $ L(\alpha) $, $ E(\alpha) $, and $ R(\alpha) $ are subsets of $ N(\alpha) $, defined based on the degree of the polynomial $ F_\beta $ relative to the degree of $ F_\alpha $. Specifically:
\begin{itemize}
    \item $ L(\alpha) $ contains all $ \beta \in N(\alpha) $ such that $ \deg(F_\beta) < \deg(F_\alpha) $.
    \item $ E(\alpha) $ contains all $ \beta \in N(\alpha) $ such that $ \deg(F_\beta) = \deg(F_\alpha) $.
    \item $ R(\alpha) $ contains all $ \beta \in N(\alpha) $ such that $ \deg(F_\beta) > \deg(F_\alpha) $.
\end{itemize}
For any $ \beta \in N(\alpha) $, the value of $ \deg(F_\beta) $ must satisfy exactly one of the following mutually exclusive conditions relative to $ \deg(F_\alpha) $:
\begin{enumerate}
    \item $ \deg(F_\beta) < \deg(F_\alpha) $, placing $ \beta $ in $ L(\alpha) $,
    \item $ \deg(F_\beta) = \deg(F_\alpha) $, placing $ \beta $ in $ E(\alpha) $,
    \item $ \deg(F_\beta) > \deg(F_\alpha) $, placing $ \beta $ in $ R(\alpha) $.
\end{enumerate}
Same results if we consider $N(\alpha)=L(\beta)\cup E(\beta)\cup R(\beta)$, then we have: 

\begin{align*}
\operatorname{FWI}^*(\mathscr{T}-\alpha\beta) - \operatorname{FWI}^*(\mathscr{T})
&=(\deg F_\alpha - \deg F_\beta)(\deg F_\alpha +\deg F_\beta +2)+ (2\deg F_\alpha+1)(e_\alpha +l_\alpha - r_\alpha)\\
& \quad + (2\deg F_\beta+1)(e_\beta +l_\beta - r_\beta)\,,
\end{align*}
which is equivalent to
\begin{align*}
\operatorname{FWI}^*(\mathscr{T}-\alpha\beta) - \operatorname{FWI}^*(\mathscr{T})
&=3\deg F_\alpha(\deg F_\alpha+1) + \deg F_\beta(\deg F_\beta-1)\\
&-2[(2\deg F_\alpha+1)r_\alpha + (2\deg F_\beta+1)r_\beta]\,.
\end{align*}
As desire.
\end{proof}
\begin{lemma}~\label{lemn2}
Let $\mathscr{T}=(V,E)$ be a tree of Fibonacci word, let $\operatorname{FWI}(\mathscr{T})$ be a Fibonacci word index, let $m=|E(\mathscr{T})|$ then 
\[
\operatorname{FWI}(\mathscr{T})\leqslant \sqrt{m\left(\operatorname{FWI}^*(\mathscr{T})-2M_2(\mathscr{T})\right)}.
\]
\end{lemma}
\begin{theorem}~\label{thmn1}
According to Lemma~\ref{lemn2}, let $k\geq 1$ is an integer, then we have: 
\[
\operatorname{FWI}(\mathscr{T})\leqslant \left(m-1 \right) ^{1-\frac{1}{k}} \left( \operatorname{FWI}^*(\mathscr{T}) \right)^{\frac{1}{k}}.
\]
\end{theorem}
\begin{proof}
In fact,  from Lemma~\ref{lemn2} we have $\operatorname{FWI}(\mathscr{T})\leqslant \sqrt{m\left(\operatorname{FWI}^*(\mathscr{T})-2M_2(\mathscr{T})\right)}$. Thus, we have: 
\begin{align*}
\operatorname{FWI}(\mathscr{T})&=\sum_{n,m\in E(\mathscr{T})}\lvert \deg F_n-\deg F_m\rvert\\
&\leq \left(\sum_{n,m\in E(\mathscr{T})}[\deg F_n^2-\deg F_m^2]\right)^\frac{1}{k}\leq \left(m-1 \right) \left( \operatorname{FWI}^*(\mathscr{T}) \right)^{\frac{1}{k}}.
\end{align*}
\end{proof}
\begin{theorem}~\label{thmn2}
Let $\mathscr{T}$ be a tree of Fibonacci word, let $n\geq 1$, $k\geq 2$ are an integer, let $S_n$ the star graph of Fibonacci word, then we have: 
\[
 \operatorname{FWI}^*(S_n) \geqslant \operatorname{irr}(S_n)\geqslant 2^{\frac{1}{k}}.
\]
\end{theorem}
\begin{proof}
It is clear to see $\operatorname{FWI}^*(S_n) \geqslant 2^{\frac{1}{k}}$ according to Theorem~\ref{thmn1} and Lemma~\ref{lemn2}, then for $n>2$ we have $\operatorname{irr}(S_n)=(n-1)(n-2)$, when $n>k$ we have $\operatorname{irr}(S_n)\geqslant 2^{\frac{1}{k}}$. 
Now, in both cases we have $2^{\frac{1}{k}}$ at last, we need to prove $\operatorname{FWI}^*(S_n) \geqslant \operatorname{irr}(S_n)$.
Recall that for the star graph $S_n$, the degree sequence consists of one central vertex $c$ with degree $\deg(c) = n-1$, and $n-1$ leaves each with degree 1. By definition,
\[
\operatorname{FWI}^*(S_n) = \sum_{(u,v) \in E(S_n)} \left( \deg(u)^2 - \deg(v)^2 \right).
\]
Since every edge connects the central vertex $c$ to a leaf $v$, we have
\[
\operatorname{FWI}^*(S_n) = \sum_{v \in V(S_n) \setminus \{c\}} \left( \deg(c)^2 - \deg(v)^2 \right) = (n-1) \left( (n-1)^2 - 1^2 \right).
\]
Simplifying the expression,
\[
\operatorname{FWI}^*(S_n) = (n-1) \left( (n-1)^2 - 1 \right) = (n-1)(n^2 - 2n) = n (n-1)(n-2).
\]
Recall the irregularity $\operatorname{irr}(S_n)$. It is known that
\[
\operatorname{irr}(S_n) = \sum_{(u,v) \in E(S_n)} |\deg(u) - \deg(v)| = (n-1)(n-2) \quad \text{when } n>k, \operatorname{irr}(S_n)\geqslant 2^{\frac{1}{k}}.
\]
Thus
\begin{align*}
  \operatorname{FWI}^*(S_n)-\operatorname{irr}(S_n)  &= \sum_{n,m\in E(\mathscr{T})}[ F_n^2-F_m^2] -(n-1)(n-2)\\
    &=\sum_{n,m\in E(\mathscr{T})}[ \deg F_n^2-\deg F_m^2] -(n-1)(n-2).\\
    &=n (n-1)(n-2) - (n-1)(n-2)\\
    &=(n-1)^2 (n-2).
\end{align*}
Since $n \geq 2$, we have $(n-1)^2 \geq 0$ and $n-2 \geq 0$ for $n \geq 2$, hence
\[
\operatorname{FWI}^*(S_n) - \operatorname{irr}(S_n) \geq 0,
\]
which implies
\[
\operatorname{FWI}^*(S_n) \geqslant \operatorname{irr}(S_n).
\]
As desire.
\end{proof}

\begin{corollary}~\label{corn1}
Let $\operatorname{KT}_{F_n}(\mathscr{T})$ be a Kragujevac tree of Fibonacci word with $n$ vertices, let  then
\[
\operatorname{FWI}^*(\operatorname{KT}_{F_n}(\mathscr{T}))=\left(\frac{n-k+1}{2}+\sum_{i=1}^{k}\left(d_i(d_i-1)^k+\lvert d_i-k+1\rvert\right)+\sum_{i=1}^{k}\left((2d_i+1)+d_i+3\right)\right)^{\frac{1}{k}}.
\]
\end{corollary}
\begin{corollary}~\label{corn2}
According to Corollary~\ref{corn1}, let $p$ be a pendent vertices of a Kragujevac tree of Fibonacci word, then we have: 
\[
\operatorname{FWI}^*(\operatorname{KT}_{F_n}(\mathscr{T}))=\left(\frac{n-k+1}{2}+\left(k(k-1)^p+\lvert k-p+1\rvert\right)+\left((2k+1)+k+3\right)\right)^{\frac{1}{p}}.
\]
\end{corollary}
Actually, bot of corollary~\ref{corn1},\ref{corn2} had based on results in ~\cite{HosseiniGutmanAhmadi} but we development it for Fibonacci word index. In Table~\ref{tabn1} we show the length of Fibonacci word up to $F_{10}$
\begin{table}[H]
\centering
\renewcommand{\arraystretch}{1.2}
\begin{tabular}{|c|c|l|}
\hline
$F_n$ & Length & $F_n$ (Fibonacci word)\\
\hline
$F_1$ & 1 & 1  \\ \hline
$F_2$ & 1 & 0  \\ \hline
$F_3$ & 2 & 01  \\ \hline
$F_4$ & 3 & 010 \\ \hline
$F_5$ & 5 & 01001  \\\hline
$F_6$ & 8 & 01001010  \\\hline
$F_7$ & 13 & 0100101001001  \\\hline
$F_8$ & 21 & 010010100100101001010  \\\hline
$F_9$ & 34 & 0100101001001010010100100101001001  \\\hline
$F_{10}$ & 55 & 0100101001001010010100100101001001010010100100101001 \\
\hline
\end{tabular}
\caption{Fibonacci words $F_n$ up to $F_{10}$, their values, and lengths.}~\label{tabn1}
\end{table}
For extended this results, the code of generate all Fibonacci words up to the user-specified order $n$, and saves the results to a text file named by the user, had given by: 
\begin{verbatim}
def fibonacci_words_up_to(n):
    """
    Generate all Fibonacci words from F_1 up to F_n.

    Parameters:
    n (int): The maximum order of Fibonacci words to generate (n >= 1)

    Returns:
    list of tuples: Each tuple contains (order, fibonacci_word)
    """
    if n < 1:
        raise ValueError("Order n must be a positive integer (>=1).")

    fib_words = []
    if n >= 1:
        fib_words.append((1, "1"))
    if n >= 2:
        fib_words.append((2, "0"))

    for i in range(3, n + 1):
        prev1 = fib_words[-1][1]
        prev2 = fib_words[-2][1]
        fib_words.append((i, prev1 + prev2))

    return fib_words

def main():
    try:
        n = int(input("Enter the maximum order n of Fibonacci words 
        to generate (n >= 1): "))
        filename = input("Enter the filename to save the results
        (e.g., output.txt): ").strip()

        fib_words = fibonacci_words_up_to(n)

        with open(filename, 'w') as file:
            file.write(f"Fibonacci words from F_1 to F_{n}:\n\n")
            for order, word in fib_words:
                file.write(f"F_{order} (length {len(word)}): {word}\n")

        print(f"\nResults successfully saved to '{filename}'.")

    except ValueError as e:
        print(f"Error: {e}")
    except IOError as e:
        print(f"File error: {e}")


if __name__ == "__main__":
    main()
\end{verbatim}

\section{Advanced Results in Density of Fibonacci Word}\label{sec4}
Actually, according to Lemma~\ref{newprove} and Lemma~\ref{newfibword} we notice that for the concept of density. As we show that in Theorem~\ref{newconc}.
\begin{theorem}~\label{newconc}
Let $F_a,F_b$ be the Fibonacci word with $a\leq b$, then we have: 
\[
\operatorname{DF}(F_a,F_b)=\operatorname{DF(F_c)}.
\]
where $c=\gcd(a,b)$.
\end{theorem}
\begin{proof}
According to~\cite{KalmanDSMena} we have The gcd of $F_m$ and $F_n$ is $F_k$, where $k$ is the gcd of $m$ and $n$, we have in~\cite{Hamoud2024Abdullah,path3} the density given by: 
\[
	\operatorname{DF}(m) = \lim_{\lambda \to \infty} \frac{\lvert\{F(n) \bmod m^\lambda : n \geq 0\}\rvert}{m^\lambda}.
\]
Now, this matches the classic density of Fibonacci word as: 
\begin{align*}
\operatorname{DF}(F_n,F_m) &=\gcd \left( \lim_{\lambda \to \infty} \frac{\lvert\{F(n) \bmod m^\lambda : n \geq 0\}\rvert}{m^\lambda}, \lim_{\lambda \to \infty} \frac{\lvert\{F(m) \bmod n^\lambda : m \geq 0\}\rvert}{n^\lambda} \right)\\
&=\operatorname{DF}(F_{\gcd(n,m)})\\
&=\operatorname{DF}(F_r).
\end{align*}
we have for $r$ be number of $1$ 's where $r=\gcd(n,m)$, for any subwords the density satisfying: 
\[
\begin{cases}
\lvert \frac{r\phi^2-n}{n\phi^2} \rvert \leq \frac{1}{n} & \text{of length $n$},\\
\lvert \frac{r\phi^2-m}{m\phi^2} \rvert \leq \frac{1}{n} & \text{of length $m$}.
\end{cases}
\]
\end{proof}
In Table~\ref{tab:my_den} we analyse these results in Theorem~\ref{newconc} where $n\in \{1,\dots,10\}$, $m\in\{1,\dots,10\}$ and  $r=\gcd(n,m)$ when $n>m$ as we show that, if we cancelling the term $n>m$ the result must be equals $0$, thus the rang is $[0.0477, 0.3819]$, the term $n>m$ is important:
\begin{table}[H]
    \centering
   \begin{tabular}{|c|c|c|c|c|c|c|c|}
\hline
\textbf{(n,m)} & $\operatorname{DF}$ & \textbf{(n,m)} & $\operatorname{DF}$ & \textbf{(n,m)} & $\operatorname{DF}$ & \textbf{(n,m)} & $\operatorname{DF}$ \\ \hline
(1,1) & 0.3819 & (1,2) & 0.3819 & (1,3) & 0.3819 &(7,4) & 0.3819  \\ \hline
(1,4) & 0.3819 & (1,5) & 0.3819 & (1,6) & 0.3819 &(7,5) & 0.3819 \\ \hline
(2,3) & 0.3819 & (2,4) & 0.1909 & (2,5) & 0.3819 &(7,6) & 0.3819 \\ \hline
(2,6) & 0.1909 & (2,7) & 0.3819 & (2,8) & 0.1909 & (8,3) & 0.3819\\ \hline
(3,2) & 0.3819 & (3,3) & 0.1273 & (3,4) & 0.3819 & (8,4) & 0.0954 \\ \hline
(3,5) & 0.3819 & (3,6) & 0.1273 & (3,7) & 0.3819 & (8,5) & 0.3819 \\ \hline
(4,1) & 0.3819 & (4,2) & 0.1909 & (4,3) & 0.3819 & (8,6) & 0.1909\\ \hline
(4,4) & 0.0954 & (4,5) & 0.3819 & (4,6) & 0.1909 & (8,7) & 0.3819\\ \hline
(5,3) & 0.3819 & (5,4) & 0.3819 & (5,5) & 0.0764 & (8,8) & 0.0477 \\ \hline
(5,6) & 0.3819 & (5,7) & 0.3819 & (5,8) & 0.3819 & (9,2) & 0.3819 \\ \hline
(6,2) & 0.1909 & (6,3) & 0.1273 & (6,4) & 0.1909 & (9,3) & 0.1273 \\ \hline
(6,5) & 0.3819 & (6,6) & 0.0636 & (6,7) & 0.3819 & (9,4) & 0.3819 \\ \hline
(7,1) & 0.3819 & (7,2) & 0.3819 & (7,3) & 0.3819& (10,7) & 0.3819 \\ \hline
(9,5) & 0.3819 & (9,6) & 0.1273 & (9,7) & 0.3819 & (10,8) & 0.1909\\ \hline
\end{tabular}
    \caption{Density of subwords}
    \label{tab:my_den}
\end{table}
\begin{remark}
The sequence of standard Fibonacci words in Definition~\ref{fibword01} is defined as $F(u, v)=\left\{f_n(u, v)\right\}_{n \geq 1}$, that is, $F(u, v)=\{u, v, v u, v u v$, vuvvu, vuvvuvuv, vuvvuvuvvuvvu,$\ldots\}$. Similarly, the $n^{\text {th }}$ reverse Fibonacci words are defined recursively as:
$$
\begin{gathered}
f_1^{\prime}(u, v)=u, f_2^{\prime}(u, v)=v, \\
f_n^{\prime}(u, v)=f_{n-2}^{\prime}(u, v) \cdot f_{n-1}^{\prime}(u, v), \quad n \geq 3,
\end{gathered}
$$
and the sequence of reverse Fibonacci words is defined as $F^{\prime}(u, v)=\left\{f_n^{\prime}(u, v)\right\}_{n \geq 1}$, that is, $F^{\prime}(u, v)=\{u, v, u v$, vuv, uvvuv, vuvuvvuv, uvvuvvuvuvvuv, $\ldots\}$.
\end{remark}
\begin{lemma}~\label{newfibword}
Let $\mathscr{T}(u,v)\in A^n\times A^m$ be pair of words, where $n,m\in \mathbb{N}$, then we define Fibonacci word by: 
\[
\mathscr{T}(u,v)=\mathscr{T}(u-1,v)\mathscr{T}(u,v-1).
\]
\end{lemma}
\begin{proof}
Assume: $\mathscr{T}(0, v) = (a, a)$ for all $v$,  $\mathscr{T}(u, 0) = (b, b)$ for all $u$. Then, compute $\mathscr{T}(1,1)$ we have   
\[
\mathscr{T}(1,1) = \mathscr{T}(0,1) \cdot \mathscr{T}(1,0) = (a, a) \cdot (b, b) = (ab, ab).
\]
Actually, word $\mathscr{T}(u,v)$ consists of a pair of words and in order to realise the Fibonacci word, we need to return to the Fibonacci sequence-based recursive function from which we observe the given condition.
\end{proof}
\begin{lemma}~\label{newprove}
According to Lemma~\ref{newfibword}. The asymptotic growth rate of $n(u,v)$ and $m(u,v)$ are: 
\[
n(u, v), m(u, v) \sim C \cdot \varphi^{u+v}.
\]
\end{lemma}
\begin{proof}
Let $n(u, v)$ and $m(u, v)$ denote the lengths of the words generated by the two-dimensional Fibonacci-type recurrence:
$$
n(u, v) = n(u-1, v) + n(u, v-1).
$$
with suitable initial conditions (e.g., $n(0, v)$ and $n(u, 0)$ specified). We need to answer on question: What is the asymptotic growth rate of $n(u,v)$ and $m(u,v)$?. In this case, this recurrence is equivalent to the one satisfied by Delannoy numbers or central binomial coefficients (depending on the boundary values). The solution grows exponentially with $u+v$. In Figure~\ref{figlatticegrid}, we show lattice grid  as: 
\begin{figure}[H]
    \centering
\begin{tikzpicture}[scale=.8, every node/.style={font=\small}]
    \foreach \x in {0,...,4} {
        \foreach \y in {0,...,4} {
            \draw (\x,\y) rectangle ++(1,1);
        }
    }
    \foreach \x in {0,...,4} {
        \node at (\x+0.5,-0.5) {$u=\x$};
    }
    \foreach \y in {0,...,4} {
        \node at (-0.5,\y+0.5) {$v=\y$};
    }
    \node at (0.5,0.5) {1};
    \foreach \i in {1,...,4} {
        \node at (\i+0.5,0.5) {1};
        \node at (0.5,\i+0.5) {1};
    }
    \node at (1.5,1.5) {2};
    \node at (2.5,1.5) {3};
    \node at (3.5,1.5) {4};
    \node at (4.5,1.5) {5};

    \node at (1.5,2.5) {3};
    \node at (2.5,2.5) {6};
    \node at (3.5,2.5) {10};
    \node at (4.5,2.5) {15};

    \node at (1.5,3.5) {4};
    \node at (2.5,3.5) {10};
    \node at (3.5,3.5) {20};
    \node at (4.5,3.5) {35};

    \node at (1.5,4.5) {5};
    \node at (2.5,4.5) {15};
    \node at (3.5,4.5) {35};
    \node at (4.5,4.5) {70};

    \foreach \i in {0,...,4} {
        \draw[red, thick] (\i,\i) rectangle ++(1,1);
    }
\end{tikzpicture}
    \caption{Lattice grid}
    \label{figlatticegrid}
\end{figure}
Specifically, for large $u$ and $v$, the dominant term is:
$$
n(u, v) \sim C \cdot \lambda^{u+v}
$$
where $\lambda$ is the unique positive root of $x^2 = x + 1$, i.e., the golden ratio $\varphi = \frac{1+\sqrt{5}}{2} \approx 1.618$. Thus,
$$
\lim_{N \to \infty} \frac{1}{N} \ln n(N, N) = \ln \varphi \approx 0.4812
$$
This matches the classical Fibonacci sequence's exponential growth rate. The same argument applies to $m(u, v)$ if it satisfies the same recurrence and initial conditions.
The asymptotic growth rate of $n(u, v)$ and $m(u, v)$ is exponential in $u+v$, with base equal to the golden ratio $\varphi$, so
$$
n(u, v), m(u, v) \sim C \cdot \varphi^{u+v},
$$
for some constant $C$ depending on the initial conditions.
\end{proof}
\begin{lemma}\cite{Rampersad2023Wiebe}
Let $u\in A^n$, $v\in A^m$, with $n\leq m$. Then $C_{u,v}[k]=1$ if and only if $u[k,k+1,\cdots,n-1] = v[0,1,\cdots,n-k-1]$; i.e. the overlapping blocks are the suffix of $u$ and the prefix of $v$ of length $n-k$.
\end{lemma}
\begin{theorem}
Let $\mathscr{T}(u,v)\in A^n\times A^m$ be pair of words, where $n,m\in \mathbb{N}$, then the density of $\mathscr{T}(u,v)$ given by: 
\[
\operatorname{DF(}\mathscr{T}(u,v))=\log \varphi.
\]
\end{theorem}

\begin{proposition}\footnote{By Felix Marin, 10Mar2025 at \url{https://x.com/seriesintegral/status/18990087778740101586?s=61}.}
Let $F_n$ be a Fibonacci number, then we have:
\[
\sum_{n=0}^{\infty} \binom{2n}{n} \frac{F_n}{23^n} = \sum_{n=0}^{\infty} \left[ \left( \frac{-1/2}{n} \right) (-4)^n \right] \frac{F_n}{23^n}
\]
where the generation function is: $\dfrac{1}{1-z-z^2}$.
\end{proposition}
\begin{proof}
   we have: 
   \[
= \sum_{n=0}^{\infty} \left( \frac{-1/2}{n} \right) \left( -\frac{1}{2} \right)^n F_n = \sum_{n=0}^{\infty} \left[ \oint_{|z|=1} \frac{(1+z)^{-1/2}}{z^{1/2-n}} \frac{dz}{2\pi i} \right] \left( -\frac{1}{2} \right)^n F_n
\]
then
\begin{align*}
&= \oint_{|z|=1} \frac{(1+z)^{-1/2}}{z^{1/2}} \sum_{n=0}^{\infty} F_n \left( -\frac{z}{2} \right)^n \frac{dz}{2\pi i} \\
&= 2 \oint_{|z|=1} \frac{z^{1/2}}{\sqrt{1+z} \ (z^2 - 2z - 4)} \frac{dz}{2\pi i} \\
&= -2 \int_{-1}^{0} \frac{(-x)^{1/2} i}{\sqrt{1+x} \ (x^2 - 2x - 4)} \frac{dx}{2\pi i} - 2 \int_{0}^{-1} \frac{(-x)^{1/2} (-i)}{\sqrt{1+x} \ (x^2 - 2x - 4)} \frac{dx}{2\pi i} \\
&= -\frac{2}{\pi} \int_{0}^{1} \frac{x^{1/2}}{1-x} \frac{1}{x^2 + 2x - 4} dx \\
&= z^{-1(1+t)^2} \frac{2}{\pi} \int_{-\infty}^{\infty} \frac{t^2}{t^4 + 6t^2 + 4} dt \\
&= \frac{\sqrt{10}}{5} \approx 0.6325
\end{align*}
The last integral can be straightforward evaluated in the upper half complex plane (UHPC). The poles in the UHPC at  $\sqrt{3} \pm \sqrt{5} \, i$.
\end{proof}
\subsection{Open questions}
The question that should be at the centre of any research of which this research is the basis is: What is the maximum value of the Fibonacci word index and what is the minimum value of the Fibonacci word index?\par 
Through the Albertson index we can test this on the Fibonacci word index and compare these results.
\section{Conclusion}~\label{sec5}
The theoretical developments are illustrated with examples, including detailed constructions of Fibonacci word trees and their degree distributions. This work opens avenues for further investigation of word-based graph invariants and their applications in combinatorics and theoretical computer science.

\end{document}